\documentclass[11pt,3p, times,  reqno]{amsart}

 \usepackage{hyperref} %for contents
 \usepackage{amssymb}
 \usepackage{amsthm}
 \usepackage{mathrsfs}
 \usepackage{amsmath}
  \usepackage{lineno}
\usepackage[right = 2.5cm, left=2.5cm, top = 2.5cm, bottom =2.5cm]{geometry}
\newcommand{\be}{\begin{equation}}
\newcommand{\ee}{\end{equation}}
\newcommand{\bes}{\begin{equation}\begin{aligned}}
\newcommand{\ees}{\end{aligned}\end{equation}}
\newcommand{\ben}{\begin{equation}\nonumber\begin{aligned}}

\renewcommand{\leq}{\leqslant}
\renewcommand{\geq}{\geqslant}

 \newtheorem{theorem}{Theorem}[section]
\newtheorem{lemma}[theorem]{Lemma}

\newtheorem{remark}[theorem]{Remark}

\numberwithin{equation}{section}


\author{Renhai  Wang}
\address[Renhai  Wang]
{  Institute of Applied Physics and Computational Mathematics, Beijing 100088, China}
\email[R.~Wang]{rwang-math@outlook.com}

\author{Tom\'as Caraballo} \address[Tom\'as Caraballo]
{Departamento de Ecuaciones Diferenciales y An\'{a}lisis Num\'{e}rico, Facultad de
     	Matem\'{a}ticas, \\ \small Universidad de Sevilla, C/ Tarfia s/n, 41012-Sevilla, Spain}
\email[T. Caraballo]{caraball@us.es}

\author{Nguyen Huy Tuan$^*$} \address[Nguyen Huy Tuan]
{Division of Applied Mathematics, Science and Technology Advanced Institute, Van Lang University, Ho Chi Minh City, Vietnam, \\ \small Faculty of Technology, Van Lang University, Ho Chi Minh City, Vietnam}
\email[N.H. Tuan]{nguyenhuytuan@vlu.edu.vn}

\thanks{$^*$Corresponding author: Nguyen Huy Tuan (E-mail: nguyenhuytuan@vlu.edu.vn)}

\begin{document}

\baselineskip=1.3\baselineskip

\begin{abstract}
\textsf{The limiting stability of invariant probability measures of time homogeneous
transition semigroups for autonomous stochastic systems has been extensively discussed in the literature. In this paper we
initially initiate a program to study the asymptotic stability of evolution systems of probability measures of time inhomogeneous transition operators for nonautonomous stochastic systems. A general theoretical result on
this topic is established in a Polish space by establishing some sufficient conditions which can be verified in applications. Our abstract results are applied to a stochastic lattice reaction-diffusion equation driven by a time-dependent nonlinear noise. A time-average method and a mild condition on the time-dependent diffusion function are used to prove that
 the limit of every evolution system of probability measures must be an evolution system of probability measures of the limiting equation. The theoretical results are expected to  be applied to various stochastic lattice systems/ODEs/PDEs in the future.}
\end{abstract}
\title[Asymptotic stability of evolution systems of probability measures]{Asymptotic stability of evolution systems of probability measures for nonautonomous stochastic systems: Theoretical results and applications}
\maketitle

\section{Introduction}
\subsection{Statement of problems}
A basic approach to look at the asymptotic stability of stochastic systems with noise perturbations is to consider the limiting stability of their invariant probability measures with respect to the noise intensity. Recently, this kind of the limiting stability of
invariant probability measures of time
homogeneous transition semigroups was discussed  in the literature for \emph{autonomous} stochastic lattice systems/ODEs/PDEs, see e.g., \cite{Chenzhang1,Chenzhang2,LIding0,LIding1}.
As far as the authors can find, currently,  there are no results reported in the literature on
the limiting stability of an \emph{evolution system of probability measures} (an extension of invariant measures from autonomous to nonautonomous developed by Da Prato and
R\"{o}ckner \cite{DaPrato3,DaPrato3.1}) of time
inhomogeneous transition  operators for \emph{nonautonomous} stochastic systems.

\subsection{General framework and  theoretical results}The goal of the present work is to initiate a program of studying limiting stability of evolution systems of probability measures of time inhomogeneous transition operators. We will establish a general setting on the limiting stability of
evolution system of probability measures for an abstract time inhomogeneous transition operator on a Polish space.
More specifically,
we will show, under certain conditions,  that the limit of every evolution system of probability measures (if exists) must be an evolution system of probability measures of the limiting transition operator.

Let $(\mathbb{X},\|\cdot\|_\mathbb{X})$ be a Polish space, $\mathcal{P}(\mathbb{X})$ be the set of all
probability measures on $\mathbb{X}$, $\mathscr{B}(\mathbb{X})$ be
the Borel $O$-algebra on $\mathbb{X}$, and  $B_b(\mathbb{X})$ ($C_b(\mathbb{X})$) the space of all
bounded Borel (continuous) functions on $\mathbb{X}$. Assume
$\{X^\epsilon(t,\tau,x), t\geq \tau\in\mathbb{R}\}$ on $\mathbb{X}$ is a unique
noise driven stochastic process
with value $x\in\mathbb{X}$ and the noise intensity $\epsilon\in[a,b]$ for $0\leq a<b<\infty$.
For $\varphi\in B_b(\mathbb{X})$, $\Lambda\in \mathscr{B}(\mathbb{X})$ and $\eta\in\mathcal{P}(\mathbb{X})$,  we
define a transition operator
 $(P^\epsilon_{\tau,t})_{t\geq\tau}$ by $(P^\epsilon_{\tau,t}\varphi)(x)=\mathbb{E}
[\varphi(X^\epsilon(t,\tau,x)) ]$, a transition probability function
$P^\epsilon_{\tau,t}(u_0,\Lambda)
=\mathbb{P}\{\omega\in\Omega
:u^\epsilon(t,\tau,u_0)\in\Lambda\}$ and the adjoint operator  $(Q^\epsilon_{\tau,t})_{t\geq\tau}$ of $(P^\epsilon_{\tau,t})_{t\geq\tau}$  by
$Q^\epsilon_{\tau,t}\eta(\Lambda)=
\int_{\mathbb{X}
}P^\epsilon_{\tau,t}(x,\Lambda)
\eta(dx)$. As in Da Prato and
R\"{o}ckner \cite{DaPrato3,DaPrato3.1}, we say $\{\eta^\epsilon_t\}_{t\in\mathbb{R}}\subseteq
\mathcal{P}(\mathbb{X})$
is an \emph{evolution system of probability measures}  of $(P^\epsilon_{\tau,t})_{t\geq\tau}$ if
$Q^\epsilon_{\tau,t}\eta_\tau=\eta_t$ for all  $t\geq \tau\in\mathbb{R}$.

For some technical reasons, the following assumption is needed when we  discuss the limiting stability of
evolution systems of probability measures of time inhomogeneous transition operators.

\textbf{CIP}(\texttt{Convergence in
Probability}). For every
compact set $\mathbf{K}\subseteq\mathbb{X}$, $\tau\in\mathbb{R}$, $t\geq\tau$, $\epsilon_0\in[a,b]$ and $\delta>0$,
\begin{align}\label{CIP}
&\lim_{\epsilon\rightarrow\epsilon_0}
\sup_{x\in \mathbf{K}}
\mathbb{P}\bigg(\Big\{\omega\in\Omega:
\|X^\epsilon(t,\tau,x)-
X^{\epsilon_0}(t,\tau,x)\|_\mathbb{X}\geq \delta \Big\}\bigg)=0.
\end{align}

\begin{theorem}\label{Main-results}
(\textbf{Theoretical results}) Assume \textbf{CIP} holds. Let
$\{\eta^{\epsilon_0}_t
\}_{t\in\mathbb{R}}$  be a family of probability measures on $\mathbb{X}$ for $\epsilon_0\in[a,b]$, and $\{\eta^{\epsilon_n}_t
\}_{t\in\mathbb{R}}$  be an evolution system of probability measures of $(P^{\epsilon_n}_{\tau,t})_{t\geq\tau}$ on $\mathbb{X}$ for $n\in\mathbb{N}$.
If $(P^{\epsilon_0}_{\tau,t})_{t\geq\tau}$ is Feller, $\epsilon_n\rightarrow\epsilon_0$, and
$\eta^{\epsilon_n}_t\rightarrow\eta^{\epsilon_0}_t$ weakly for each $t\in\mathbb{R}$,  then
$\{\eta^{\epsilon_0}_t
\}_{t\in\mathbb{R}}$ must be an evolution system of probability measures of $(P^{\epsilon_0}_{\tau,t})_{t\geq\tau}$.
\end{theorem}

\begin{remark}
(i) We only require $(P^{\epsilon}_{\tau,t})_{t\geq\tau}$ is Feller at $\epsilon=\epsilon_0$. (ii) The convergence in probability in \eqref{CIP} is not necessary to be uniform for $t$. But for most stochastic systems, \eqref{CIP} can be proved uniformly for $t$:
\begin{align*}
&\lim_{\epsilon\rightarrow\epsilon_0}
\sup_{x\in \mathbf{K}}
\mathbb{P}\bigg(\Big\{\omega\in\Omega:
\sup_{t\in[\tau,\tau+T]}\|X^\epsilon(t,\tau,x)-
X^{\epsilon_0}(t,\tau,x)\|_\mathbb{X}\geq \delta \Big\}\bigg)=0.
\end{align*}
\end{remark}

Theorem \ref{Main-results} can be viewed as an extension version of limiting stability of invariant measures of time homogeneous transition operators from the autonomous to the nonautonomous framework.

\subsection{Application of theoretical results}
Our abstract result in Theorem \ref{Main-results}  is expected to be applied to various stochastic ODEs/PDEs/lattice systems with noise perturbations.
In particular, we apply this abstract result to a stochastic lattice reaction-diffusion equation driven by a
 \emph{nonlinear} noise on $\mathbb{Z}$
  for $t>\tau$ with $\tau \in \mathbb{R}$:
\begin{align}\label{AA1}
 & du^\epsilon_i(t)+\lambda u^\epsilon_i(t)dt-\nu\big(u^\epsilon_{i}(t)
 -2u^\epsilon_{i-1}(t) +u^\epsilon_{i+1}(t)\big)dt
+|u^\epsilon_i(t)|^{p-2}|u^\epsilon_i(t)|dt
 =\epsilon\sigma(t,u^\epsilon_i(t)) dW(t),
\end{align}
with initial condition $u^\epsilon_i(\tau)
=u_{0,i}$, where $\lambda,\nu>0$, $p>2$, $\epsilon>0$,
 $W$ is a  two-sided, real-valued
Wiener process on
$(\Omega,\mathcal{F}, \{\mathcal{F}_t\}_{t\in\mathbb{R}},
\mathbb{P})$ the
complete filtered probability space, and the
nonlinear function $\sigma(t,\cdot):
\mathbb{R}\rightarrow\mathbb{R}$ is
locally Lipschitz  such that $|\sigma(t,s)|\leq
\delta|s|+g(t)$ for a constant $\delta>0$ and a time-dependent
deterministic function $g$ satisfying certain conditions. We say the noise in
\eqref{AA1} is a time-dependent nonlinear noise just because the diffusion function $\sigma$ depends on time $t$, and is nonlinear in the unknown function $u^\epsilon_i$.

Note that lattice systems can be regarded as space discretization versions of PDEs that have many applications in the real world such as electric circuits, pattern formation, propagation of nerve pulse, chemical reaction and others.
The existence and limiting stability of invariant probability measures for \emph{autonomous} stochastic
lattice systems has been considered recently in
\cite{Chenzhang1,Chenzhang2,LIding0,LIding1,bwang}. The reader is referred to \cite{Caraball1,Caraball2,Caraball3,
Caraball4,Han1,Han2,Zhou2012jdde,Zhou2017Random} for other mathematical topics such as random attractors for stochastic
lattice systems.
In this paper we study existence and limiting stability of
evolution systems of
probability measures for the \emph{nonautonomous} stochastic
lattice system
\eqref{AA1}. To our knowledge, it seems that this is the first time to study evolution systems of
probability measures for  stochastic lattice systems.

\begin{theorem}\label{tightness}(\textbf{Existence of evolution system
of probability measures})
If  $\int_{-\infty}^{0}
e^{\lambda r}\|g(r)\|^2
dr<\infty$, then
the transition operator $(P^\epsilon_{\tau,t})_{t\geq\tau}$ for
\eqref{AA1} has an
evolution system
of probability measures $\{\mu^\epsilon_t
\}_{t\in\mathbb{R}}$ on $\ell^2$ for any $\epsilon\in\big[0,\frac{\sqrt{\lambda}}
{2\delta}\big]$.
\end{theorem}

Let
$\mathcal{E}^\epsilon_t=\big\{
\mu^{\epsilon}_t:
\{\mu^\epsilon_t\}_{t\in\mathbb{R}} \ \mbox{is an evolution system of probability measures of $(P^{\epsilon}_{\tau,t})_{t\geq\tau}$ for
\eqref{AA1}}
\big\}$ for $t\in\mathbb{R}$ and $\epsilon\in\big[0,\frac{\sqrt{\lambda}}
{2\delta}\big]$. Then we discuss the
tightness of  $\bigcup_{\epsilon\in\big[0,\frac{\sqrt{\lambda}}
{2\delta}\big]}\mathcal{E}^\epsilon_t$ and limiting stability of any sequences of  $\mathcal{E}^\epsilon_t$.

\begin{theorem}\label{union2}(\textbf{Application of theoretical results})
If $\int_{-\infty}^{t}
e^{\lambda r}\|g(r)\|^2
dr<\infty$ for each
$t\in\mathbb{R}$, then we have the following two results.

(i) The union $\bigcup_{\epsilon\in\big[0,\frac{\sqrt{\lambda}}
{2\delta}\big]
}\mathcal{E}_t^\epsilon$ is tight on $\ell^2$ for each $t\in\mathbb{R}$.

(ii)If $\mu^{\epsilon_n}_t\in
\mathcal{E}_t^{\epsilon_n}$ with $\epsilon_n\rightarrow\epsilon_0\in\big[0,\frac{\sqrt{\lambda}}
{2\delta}\big]$, then there exists $\mu^{\epsilon_0}_t\in
\mathcal{E}^{\epsilon_0}_t$  and a subsequence (not relabeled)
such that $\mu^{\epsilon_n}_t\rightarrow\mu^{\epsilon_0}_t$ weakly for each $t\in\mathbb{R}$.
\end{theorem}

\section{Proof of Theorem \ref{Main-results}
}

\begin{proof}\textbf{Proof of Theorem \ref{Main-results}}.
Given $\tau\in\mathbb{R}$, it is sufficient to show that for any
$\varphi\in C_b(\mathbb{X})$,
\begin{align}\label{rrrr}
&\int_\mathbb{X}\varphi(x) (Q^{\epsilon_0}_{\tau,t}
\eta^{\epsilon_0}_\tau)(dx)=
\int_\mathbb{X} (P^{\epsilon_0}_{\tau,t}\varphi)(x)
\eta^{\epsilon_0}_\tau(dx)=\int_\mathbb{X}
\varphi(x)
\eta^{\epsilon_0}_t(dx), \ \ \forall\  t\geq \tau.
\end{align}
The first equality is obvious. To prove the second one in  \eqref{rrrr} we consider the following  equality:
\begin{align}\label{ww0}
\int_\mathbb{X} (P^{\epsilon_0}_{\tau,t}\varphi)(x)
\eta^{\epsilon_0}_\tau(dx)-\int_\mathbb{X}
\varphi(x)
\eta^{\epsilon_0}_t(dx)
&=\int_\mathbb{X}(P^{\epsilon_0}_{\tau,t}\varphi)(x)
\eta^{\epsilon_0}_\tau(dx)-
\int_\mathbb{X}(P^{\epsilon_0}_{\tau,t}\varphi)(x)
\eta^{\epsilon_n}_\tau(dx)
\nonumber\\&\ \ +
\int_\mathbb{X}(P^{\epsilon_0}_{\tau,t}\varphi)(x)
\eta^{\epsilon_n}_\tau(dx)-\int_\mathbb{X}
\varphi(x)
\eta^{\epsilon_n}_t(dx)\nonumber\\&\ \ +\int_\mathbb{X}
\varphi(x)
\eta^{\epsilon_n}_t(dx)
-\int_\mathbb{X}
\varphi(x)
\eta^{\epsilon_0}_t(dx).
\end{align}
For the last line on the right-hand side of \eqref{ww0}, since $\eta^{\epsilon_n}_t\rightarrow
\eta^{\epsilon_0}_t$ weakly, we know that for every $\gamma$, there exists $N_1\in\mathbb{N}$ such that
\begin{align}\label{ww1}
\int_\mathbb{X}
\varphi(x)
\eta^{\epsilon_n}_t(dx)
-\int_\mathbb{X}
\varphi(x)
\eta^{\epsilon_0}_t(dx)\leq\gamma \ \ \mbox{for all $n\geq N_1$}.\end{align}
For the first line on the right-hand side of \eqref{ww0}, since $(P^{\epsilon_0}_{\tau,t})_{t\geq\tau}$ is Feller and $\eta^{\epsilon_n}_\tau\rightarrow
\eta^{\epsilon_0}_\tau$ weakly, we know that for every $\gamma$, there exists $N_2\in\mathbb{N}$ such that
\begin{align}\label{ww2}
\int_\mathbb{X}(P^{\epsilon_0}_{\tau,t}\varphi)(x)
\eta^{\epsilon_0}_\tau(dx)-
\int_\mathbb{X}(P^{\epsilon_0}_{\tau,t}\varphi)(x)
\eta^{\epsilon_n}_\tau(dx)\leq\gamma \ \ \mbox{for all $n\geq N_2$}.\end{align}

It remains to consider the second line on the right-hand side of \eqref{ww0}. By assumption \textbf{CIP} we know that
for every $\gamma>0$ and $\delta>0$, there exists $N_3\in\mathbb{N}$ such that
\begin{align}\label{uniformly0ab}
\sup_{x\in \mathbf{K}}
\mathbb{P}\bigg(\Big\{\omega\in\Omega:
\|X^{\epsilon_n}(t,\tau,x)-
X^{\epsilon_0}(t,\tau,x)\|_\mathbb{X}\geq \delta \Big\}\bigg)\leq\gamma \ \ \mbox{for all $n\geq N_3$}.
\end{align}
Since $\eta^{\epsilon_n}_\tau\rightarrow
\eta^{\epsilon_0}_\tau$ weakly, by Prohorov theorem (see \cite{DaPrato1,DaPrato2})  we know $\{\eta^{\epsilon_n}_\tau\}_{n=1}^\infty$ is tight on $\mathbb{X}$. This means that for every $\gamma>0$, there exists a compact set $\mathbf{K}(\gamma,\tau)\subseteq\mathbb{X}$ such that
\begin{align}\label{uniformly0}
\eta^{\epsilon_n}_\tau(\mathbb{X}/\mathbf{K}(\gamma,\tau))
<\gamma\ \  \mbox{for all $n\in\mathbb{N}$}.
\end{align}
By $\varphi\in C_b(\mathbb{X})$ and the compactness of $\mathbf{K}$ we know that
for every $\gamma>0$, there exists $\eta>0$ such that
\begin{align}\label{uniformly0a}
|\varphi(y)-\varphi(z)|<\gamma\ \ \mbox{for all $y,z\in \mathbf{K}$\ with \ $\|y-z\|_\mathbb{X}\leq\eta$}.
\end{align}
Indeed, if \eqref{uniformly0a} is false, then
there exist $\gamma_0>0$, $y_n, z_n\in\mathbf{K}$
with $\|y_n-z_n\|_\mathbb{X}\leq1/n$ such that
$|\varphi(y_n)-\varphi(z_n)|\geq\gamma_0$. Since
$\mathbf{K}$ is a compact subset of $\mathbb{X}$, there exist $x\in\mathbf{K}$ and a subsequence $\{n_k\}_{k=1}^\infty\subseteq \{n\}_{n=1}^\infty$ such that $\lim_{k\rightarrow\infty}
\|y_{n_k}-x\|_\mathbb{X}=0$
. Then one can verify $\lim_{k\rightarrow\infty}
\|z_{n_k}-x\|_\mathbb{X}=0$.
Since $\varphi$ is continuous, letting $k\rightarrow\infty$ in
$\gamma_0\leq|\varphi(y_{n_k})-\varphi(z_{n_k})|
\leq|\varphi(y_{n_k})-\varphi(x)|+
|\varphi(z_{n_k})-\varphi(x)|$, we find a  contradiction $\gamma_0\leq0$.

Since $\{\eta^{\epsilon_n}_t
\}_{t\in\mathbb{R}}$ is an evolution system of probability measures of $(P^{\epsilon_n}_{\tau,t})_{t\geq\tau}$ on $\mathbb{X}$, by \eqref{uniformly0ab}-\eqref{uniformly0a} we find that
for all $n\geq N:=\max\{N_1,N_2,N_3\}$,
\begin{align}\label{uniformly1}
&\int_\mathbb{X}(P^{\epsilon_0}_{\tau,t}\varphi)(x)
\eta^{\epsilon_n}_\tau(dx)-\int_\mathbb{X}
\varphi(x)
\eta^{\epsilon_n}_t(dx)
\nonumber\\&=\int_\mathbb{X} \mathbb{E}\big[
\varphi(X^{\epsilon_0}(t,\tau, x) )\big]
\eta^{\epsilon_n}_\tau(dx)-\int_\mathbb{X}
\mathbb{E}\big[
\varphi(X^{\epsilon_n}(t,\tau, x) )\big]
\eta^{\epsilon_n}_\tau(dx)
\nonumber\\&\leq\int_\mathbb{X} \mathbb{E}\big[|
\varphi(X^{\epsilon_n}(t,\tau, x) )-\varphi(X^{\epsilon_0}(t,\tau, x) )|\big]
\eta^{\epsilon_n}_\tau(dx)
\nonumber\\&\leq\int_\mathbb{\mathbf{K}(\gamma,\tau)} \mathbb{E}\big[|
\varphi(X^{\epsilon_n}(t,\tau, x) )-\varphi(X^{\epsilon_0}(t,\tau, x) )|\big]
\eta^{\epsilon_n}_\tau(dx)+2 \sup_{x\in\mathbb{X}}|\varphi(x)|
\eta^{\epsilon_n}_\tau(\mathbb{X}/
\mathbf{K}(\gamma,\tau))
\nonumber\\&\leq\int_\mathbb{\mathbf{K}(\gamma,\tau)} \mathbb{E}\big[|
\varphi(X^{\epsilon_n}(t,\tau, x) )-\varphi(X^{\epsilon_0}(t,\tau, x) )|\big]
\eta^{\epsilon_n}_\tau(dx)+2\gamma \sup_{x\in\mathbb{X}}|\varphi(x)|
\nonumber\\&\leq\int_\mathbb{\mathbf{K}(\gamma,\tau)} \int_{\{\omega\in\Omega:\|
X^{\epsilon_n}(t,\tau, x) -X^{\epsilon_0}(t,\tau, x) \|_\mathbb{X}\geq\eta\}}\big|
\varphi(X^{\epsilon_n}(t,\tau, x) )-\varphi(X^{\epsilon_0}(t,\tau, x) )\big|\mathbb{P}(d\omega)
\eta^{\epsilon_n}_\tau(dx)
\nonumber\\&\ \ +\int_\mathbb{\mathbf{K}(\gamma,\tau)} \int_{\{\omega\in\Omega:\|
X^{\epsilon_n}(t,\tau, x) -X^{\epsilon_0}(t,\tau, x) \|_\mathbb{X}<\eta\}}\big|
\varphi(X^{\epsilon_n}(t,\tau, x) )-\varphi(X^{\epsilon_0}(t,\tau, x) )\big|\mathbb{P}(d\omega)
\eta^{\epsilon_n}_\tau(dx)
\nonumber\\&\ \  +2\gamma \sup_{x\in\mathbb{X}}|\varphi(x)|
\nonumber\\&\leq2 \sup_{x\in\mathbb{X}}|\varphi(x)|\eta^{\epsilon_n}_\tau
(\mathbf{K}(\gamma,\tau))\mathbb{P}\bigg(\Big\{\omega\in\Omega:
\|u^\epsilon(t,\tau,x)-
u^{\epsilon_0}(t,\tau,x)\|_\mathbb{X}\geq\eta \Big\}\bigg)\nonumber\\&+\gamma\eta^{\epsilon_n}_\tau
(\mathbf{K}(\gamma,\tau))\mathbb{P}\bigg(\Big\{\omega\in\Omega:
\|u^\epsilon(t,\tau,x)-
u^{\epsilon_0}(t,\tau,x)\|_\mathbb{X}<\eta \Big\}\bigg)+2\gamma \sup_{x\in\mathbb{X}}|\varphi(x)|
\nonumber\\&\leq\gamma(1+4 \sup_{x\in\mathbb{X}}|\varphi(x)|).
\end{align}

Thus, as a result
of \eqref{ww1}-\eqref{ww2} and \eqref{uniformly1} we have
\begin{align*}\label{ww0}
\int_\mathbb{X} (P^{\epsilon_0}_{\tau,t}\varphi)(x)
\eta^{\epsilon_0}_\tau(dx)-\int_\mathbb{X}
\varphi(x)
\eta^{\epsilon_0}_t(dx)\leq \gamma(3+4 \sup_{x\in\mathbb{X}}|\varphi(x)|).
\end{align*}
Since $\gamma>0$ is arbitrary, we thus complete the proof.
\end{proof}

\section{Proof of Theorem \ref{tightness}
}
Let us consider the Banach space
$\ell^q=\big\{u=(u_i)_{i
\in\mathbb{Z}}: \sum_{i\in\mathbb{Z}}|u_i|^q<+\infty \big\}$ endowed with the norm
\begin{equation*}
\|u\|_q=\left\{
          \begin{array}{ll}
            \big(\sum_{i\in\mathbb{Z}}|u_i|^q
\big)^{1/q}, & \hbox{if $q\in[1,\infty)$,}
\\[6pt]\ \sup_{i\in\mathbb{Z}}|u_i|, & \hbox{if $q=\infty$.}
          \end{array}
        \right.
\end{equation*}
In particular, we write $\|\cdot\|=\|\cdot\|_2$. It is known from Wang \cite{bwang} that
for each $\tau\in \mathbb{R}$ and $u_{0}\in L^2(\Omega,\mathcal{F}_\tau; \ell^2)$, equation \eqref{AA1} has a unique solution which is
continuous $\ell^2$-valued
$\mathcal{F}_t$-adapted Markov process $u(\cdot,\tau,u_0)\in L^2(\Omega, C([\tau,\infty), \ell^2  ))\cap L^p(\Omega, L^p(\tau,\infty);  \ell^p )) $.
Then we can prove that the transition operator
$(P^\epsilon_{\tau,t})_{t\geq\tau}$ for $u(t,\tau,u_0)$ is Feller, and the process laws hold:
$P^\epsilon_{\tau,t}=
P^\epsilon_{\tau,r}P^\epsilon_{r,t}$ and
$Q^\epsilon_{\tau,t}=
Q^\epsilon_{r,t}Q_{\tau,r}^\epsilon$,
$-\infty<\tau\leq r\leq t<+\infty$. 

Next, we derive different uniform estimates of time average of the solutions.
\begin{lemma}\label{uniformestimatesC}
If $\int_{-\infty}^{t}
e^{\lambda r}\|g(r)\|^2
dr<\infty$ for each
$t\in\mathbb{R}$, then we have the following results.

(i) For each $t\in\mathbb{R}$ and  $\mathbb{N}\ni k>-t$,
\begin{align}&\label{DD2AA32Aswwerr}
\sup_{\epsilon\in\big[0,\frac{\sqrt{\lambda}}
{2\delta}\big]}\frac{1}{k+t}\int_{-k}^{t}
\mathbb{E} [\|u^\epsilon(t,\tau,u_0)\|^{2}]d\tau +\sup_{\epsilon\in\big[0,\frac{\sqrt{\lambda}}
{2\delta}\big]}\frac{1}{k+t}\int_{-k}^{t}
\int_{\tau}^{t}
e^{\lambda(r-t)}\mathbb{E}\left
[\|u^\epsilon(r,\tau,u_0)\|^2\right]drd\tau
\nonumber\\&\leq\frac{c}{(k+t)}\mathbb{E} \big[\|u_0\|^{2}\big]+
c\int_{-\infty}^{t}e^{\lambda
(r-t)}\|g(r)\|^2dr,
\end{align}
where $c>0$ is a constant independent of
$\epsilon$, $t$, $k$ and $u_0$.

(ii) For each $t\in\mathbb{R}$
and   compact set
$\mathbf{K}$ of $\ell^2$, the  solutions satisfy
\begin{align*}
\lim_{n\rightarrow\infty}\sup_{\epsilon\in\big[0,\frac{\sqrt{\lambda}}
{2\delta}\big]}
\sup_{\mathbb{N}\ni k>-t}\sup_{u_0\in\mathbf{K}}
\frac{1}{k+t}\int_{-k}^{t}\sum_{|i|\geq n}
\mathbb{E} [|u^\epsilon(t,\tau,u_0)|^2]d\tau=0.
\end{align*}

(iii) For each $t\in\mathbb{R}$ and bounded set
$\mathbf{B}$ of $\ell^2$, the solutions satisfy
\begin{align*}
\lim_{n,k\rightarrow\infty}\sup_{\epsilon\in\big[0,\frac{\sqrt{\lambda}}
{2\delta}\big]}
\sup_{u_0\in\mathbf{B}}
\frac{1}{k+t}\int_{-k}^{t}\sum_{|i|\geq n}
\mathbb{E} [|u^\epsilon(t,\tau,u_0)|^2]d\tau=0.
\end{align*}

\end{lemma}

\begin{proof}
(i) Applying It\^{o}'s formula to \eqref{AA1}, we infer that for any $\epsilon\in\big[0,\frac{\sqrt{\lambda}}
{2\delta}\big]$,
\begin{align}\label{uniformly0aaaa}
\frac{d}{dt}\mathbb{E} \big[\|u(t)\|^2\big]
 +\frac{3}{2}\lambda\mathbb{E} \big[\|u(t)\|^2\big]
 \leq c\|g(t)\|^2,
\end{align}
where $c>0$ is a constant. Multiplying \eqref{uniformly0aaaa} by $e^{\lambda t}$ and integrating
over $(\tau,t)$, we deduce
\begin{align}\label{FFF1Ad1}
\mathbb{E} \big[\|u^\epsilon(t,\tau,u_0)\|^{2}\big]&+\frac{\lambda}{2}
\int_{\tau}^{t}
e^{\lambda(r-t)}\mathbb{E}\left
[\|u^\epsilon(r,\tau,u_0)\|^2\right]dr\nonumber\\
&\leq e^{\lambda
(\tau-t)}
\mathbb{E}\big [\|u_0\|^{2}\big]+
c\int_{-\infty}^te^{\lambda
(s-t)}\|g(s)\|^2ds.
\end{align}
Integrating \eqref{FFF1Ad1} with respect to $\tau$ over $(-k,t)$, we obtain \eqref{DD2AA32Aswwerr}.

(ii)-(iii)
By a cut-off technique as used by Wang \cite[Lemma 4.2]{bwang} (see also \cite{wang1999}) one can derive, for any $\epsilon\in\big[0,\frac{\sqrt{\lambda}}
{2\delta}\big]$,
\begin{align}\label{CC9aaa}
  \mathbb{E} \Big[\sum_{|i|\geq 2n}
 |u_i(t,\tau,u_0)|^2\Big]
& \leq e^{\lambda(\tau-t)}\sum_{|i|\geq n}
 |u_{0,i}|^2
+c\int_{\tau}^t
e^{\lambda(r-t)}\sum_{|i|\geq n}
|g_{i}(r)|^2
dr\nonumber\\&\ \ +
\frac{c}{n}\int_{\tau}^t
e^{\lambda(r-t)}\mathbb{E}\left
[\|u^\epsilon(r,\tau,u_0)\|^2\right]dr
\end{align}
Integrating \eqref{CC9aaa} with respect to $\tau$
over $(-k,t)$, we have
\begin{align*}
\frac{1}{k+t}  \int_{-k}^{t}\mathbb{E} \Big[\sum_{|i|\geq 2n}
 |u_i(t,\tau,u_0)|^2\Big]d\tau
& \leq \frac{1}{\lambda(k+t)}\sum_{|i|\geq n}
 |u_{0,i}|^2
 +c\int_{-\infty}^t
e^{\lambda(r-t)}\sum_{|i|\geq n}
|g_{i}(r)|^2
dr
\nonumber\\&\ \ +\frac{c}{n(k+t)}\int_{-k}^{t}
\int_{\tau}^{t}
e^{\lambda(r-t)}\mathbb{E}\left
[\|u^\epsilon(r,\tau,u_0)\|^2\right]drd
\tau.
\end{align*}
This along with $\int_{-\infty}^{t}
e^{\lambda r}\|g(r)\|^2
dr<\infty$ and \eqref{DD2AA32Aswwerr}
completes the proof of
(ii)-(iii).
\end{proof}

\begin{proof}\textbf{Proof of Theorem \ref{tightness}}. The proof is completed by an extended
Krylov-Bogolyubov method proposed by Da Prato and R\"{o}ckner
\cite{DaPrato3.1}.
For $n\in\mathbb{N}$, we let $\chi_{[-n,n]}$ be the characteristic
function of $[-n,n]$. Then
$u^\epsilon(t,\tau,u_0)=
\tilde{u}^\epsilon_n(t,\tau,u_0)+
 \hat{u}^\epsilon_n(t,\tau,u_0)$ with
$\tilde{u}^\epsilon_n(t,\tau,u_0)=
\big(\chi_{[-n,n]}(|i|)u^\epsilon_i(t,\tau,u_0)
\big)_{i\in\mathbb{Z}}$ and
 $\hat{u}^\epsilon_n(t,\tau,u_0)
=\big((1-\chi_{[-n,n]}(|i|))
u^\epsilon_i(t,\tau,u_0)\big)_{i\in\mathbb{Z}}$.
Given $\delta>0$, $l\in\mathbb{N}$ and   $m,k\in\mathbb{N}$ with $m\leq k$,
by (i)-(ii) of Lemma \ref{uniformestimatesC} there exist
$n_l^\delta\in \mathbb{N}$ and $C>0$
independent on $\epsilon$, $m$ and $k$  such that
\begin{align*}
\sup_{\epsilon\in\big[0,\frac{\sqrt{\lambda}}
{2\delta}\big]}\frac{1}{k-m}\int_{-k}^{-m}
\mathbb{E} \big[\|u(-m,\tau,0)\|^2\big]d\tau
\leq C\ 
\end{align*}
\mbox{and} 
\begin{align*}
\ \sup_{\epsilon\in\big[0,\frac{\sqrt{\lambda}}
{2\delta}\big]}\frac{1}{k-m}\int_{-k}^{-m}
\mathbb{E} \big[\|\hat{u}^\epsilon_{n_l^\delta
}(-m,\tau,
0)\|^2\big]d\tau<\frac{\delta}{2^{4l}}.
\end{align*}
Define $$\mathcal{Y}^\delta_l=
\bigg\{u\in\ell^2:
u_i=0\ \mbox{for } |i|>n_l^\delta
\ \mbox{and} \
 \|u\|\leq\frac{2^l
\sqrt{C}}{\sqrt{\delta}}
        \bigg\}$$ and
$$\mathcal{Z}^\delta
_l=\Big\{u\in\ell^2:
\|u-v\|\leq\frac{1}{2^l}
\ \mbox{for some $v\in \mathcal{Y}^\delta_l$}\Big\}.$$
 Define a probability measure
$\eta^\epsilon_{k,m}=\frac{1}{k-m}
\int_{-k}^{-m}\mathbb{P}\Big\{\omega\in\Omega:
u^\epsilon(-m,\tau,0)\in\cdot    \Big\}
d\tau$ on $\ell^2$.
Then by Chebychev's inequality one can verify
\begin{align}\label{CC20ff}
\eta^\epsilon_{k,m}(\ell^2\setminus
\mathcal{Z}^\delta_l)&
\leq\frac{\delta}{2^{2l}
C}\frac{1}{k-m}\int_{-k}^{-m}
\mathbb{E} \big(\|u^\epsilon(-m,\tau,0)
\|^2\big)d\tau
\nonumber\\& \ \ +\frac{2^{2l}}{k-m}
\int_{-k}^{-m}
\mathbb{E} \big[\|\hat{u}^\epsilon_{n^\delta_l}
(-m,\tau,
0)\|^2\big]d\tau
\leq\frac{\delta}{2^{2l-1}}.
\end{align}
Let $\mathcal{Z}^\delta=
\bigcap_{l=1}^\infty
\mathcal{Z}^\delta_l $.
Since one can verify that $\mathcal{Z}^\delta$
is closed and totally bounded in $\ell^2$, it is compact in $\ell^2$. By
\eqref{CC20ff} we know $\eta^\epsilon_{k,m}
(\ell^2\setminus
\mathcal{Z}^\delta)
\leq\sum_{l=1}^\infty\eta^\epsilon_{k,m}
(\ell^2\setminus\mathcal{Z}^\delta_l)
<\delta$. Then  $\{\eta^\epsilon_{k,m}\}_{\mathbb{N}\ni k\geq m
}$ is tight on $\ell^2$ for each fixed $m\in\mathbb{N}$.
Then $\{\eta^\epsilon_{k,m}\}_{\mathbb{N}\ni k\geq m
}$ is tight on $\ell^2$ for each  $m\in\mathbb{N}$. Then there exists $\eta^\epsilon_m\in \mathcal{P}(\ell^2)$
and a subsequence (not relabeled) such that
$\eta^\epsilon_{k,m}\rightarrow\eta^\epsilon_{m}$ weakly as $k\rightarrow\infty$. For each fixed $t\in\mathbb{R}$, we choose $m\in\mathbb{N}$ such that
$-m\leq t$, and define
$\mu^\epsilon_t:=Q^\epsilon_{-m,t}
\eta^\epsilon_{m}$. One can verify that this definition is independent of  the choice
  of $m$.  Then for every fixed
$t\in\mathbb{R}$ and for any $\mathbb{N}\ni m\geq-\tau\geq -t$, we have
$Q^\epsilon_{\tau,t}\mu^\epsilon_{\tau}
=Q^\epsilon
_{\tau,t}Q^\epsilon_{-m,\tau}
\eta^\epsilon_{-m}=(P^\epsilon_{-m,\tau}
P^\epsilon_{\tau,t})^*
\eta^\epsilon_{-m}=
Q^\epsilon_{-m,t}\eta_{-m}=\mu^\epsilon_t$.
This complete the proof.
\end{proof}

\section{Proof of Theorem \ref{union2}
}
\begin{proof}
\textbf{Proof of Theorem \ref{union2}}.
(i) For every $\delta>0$, $t\in\mathbb{R}$,
$\epsilon\in[a,b]$ and
$\mu_t^\epsilon\in\bigcup_{\epsilon\in[a,b]
}\mathcal{E}_t^\epsilon$, we shall find a compact set
$\mathcal{Z}^\delta(t)$ independent
of $\epsilon$
such that $\mu_t^\epsilon(\mathcal{Z}^\delta(t))
>1-\delta$. Given $t\in\mathbb{R}$ and $l\in\mathbb{N}$,
by (i) and (iii) of Lemma \ref{uniformestimatesC}
 there exists
$n^\delta_l(t)\in \mathbb{N}$,
$K_l^\delta(t)\in\mathbb{N}$ and $C(t)>0$
independent of
 $\epsilon$ and $u_0$ such that
\begin{align}\label{-CC16}
 &\sup_{k\geq K_l^\delta(t)}\sup_{\epsilon\in\big[0,\frac{\sqrt{\lambda}}
{2\delta}\big]}
\sup_{u_0\in\ell^2}\frac{1}{k+t}\int_{-k}^{t}
\mathbb{E} \big[\|u(t,\tau,u_0)\|^2\big]d\tau
\leq C(t),
\\& \label{-CC17}
\sup_{\epsilon\in\big[0,\frac{\sqrt{\lambda}}
{2\delta}\big]}
\sup_{u_0\in\ell^2}\frac{1}{k+t}\int_{-k}^{t}
\mathbb{E} \big[\|\hat{u}^\epsilon_{n^\delta_l(t)}
(t,\tau,
u_0)\|^2\big]d\tau<\frac{\delta}{2^{4l}}.
\end{align}
Define $\mathcal{Y}^\delta_l(t):=
\bigg\{u\in\ell^2:
u_i=0\ \mbox{for } |i|>n^\delta_l(t)
\ \mbox{and} \
 \|u\|\leq\frac{2^l
\sqrt{C(t)}}{\sqrt{\delta}}
        \bigg\}$, $\mathcal{Z}_l^\delta(t)
:=\Big\{u\in\ell^2:
\|u-v\|\leq\frac{1}{2^l}
\ \mbox{for some $v\in \mathcal{Y}^\delta_l
(t)$}       \Big\}$\ \mbox{and $\mathcal{Z}^\delta(t)
:=\bigcap_{l=1}^\infty
\mathcal{Z}^\delta_l(t)$}. Note that $\mathcal{Z}^\delta(t)$ is
 compact in $\ell^2$. In what follows we  prove
 $\mu_t^\epsilon(\mathcal{Z}^\delta(t))>1-\delta$.
Denote by
$\mathcal{X}^\delta_n(t):=\bigcap_{l=1}^n
\mathcal{Z}^\delta_l(t)$ for $n\in\mathbb{N}$.
Then
$\mu_t^\epsilon(\mathcal{Z}^\delta(t))=
\mu_t^\epsilon\big(\bigcap_{n=1}
^\infty\mathcal{X}^\delta_n(t)\big)=
\lim\limits_{n\rightarrow\infty}\mu_t^\epsilon\big(
\mathcal{X}^\delta_n(t)\big)$, and hence there exists $N=N(\delta,t)\in
\mathbb{N}$ such that $0\leq\mu_t^\epsilon
(\mathcal{X}^\delta_n(t))-
\mu_t^\epsilon(\mathcal{Z}^\delta(t))
\leq\delta/3$ for all $n\geq N$.  By Fubini's theorem and the definition of  $\{\mu^\epsilon_t\}_{t\in\mathbb{R}}$ of $(P^\epsilon_{\tau,t})_{t\geq\tau}$,
we find that for all $\tau\leq t\in\mathbb{R}$,
\begin{align}\label{uuua}
\eta_t^\epsilon(\mathcal{X}^\delta_N(t) )&=\int_{\ell^2}
\mathbb{P}\bigg(\Big\{\omega\in\Omega:
u^\epsilon(t,\tau,x)\in \mathcal{X}^\delta_N(t)
\Big\}\bigg)\mu_\tau^\epsilon(dx)
\nonumber\\&=\int_{\ell^2}
\mathbb{P}\bigg(\Big\{\omega\in\Omega:
u^\epsilon(t,\tau,x)\in \mathcal{X}^\delta_N(t)
\Big\}\bigg) h^\epsilon(\tau)\mu^\epsilon(dx)
\nonumber\\&=
\frac{1}{k+t}\int_{-k}^t\int_{\ell^2}
\mathbb{P}\bigg(\Big\{\omega\in\Omega:
u^\epsilon(t,\tau,x)\in \mathcal{X}^\delta_N(t)
\Big\}\bigg) h^\epsilon(\tau)\mu^\epsilon(dx)d\tau
\nonumber\\&=\int_{\ell^2} \frac{1}{k+t}\int_{-k}^t
\mathbb{P}\bigg(\Big\{\omega\in\Omega:
u^\epsilon(t,\tau,x)\in \mathcal{X}^\delta_N(t)
\Big\}\bigg) h^\epsilon(\tau)d\tau\mu^\epsilon(dx),
\end{align}
where
$\mu^\epsilon\in\mathcal{P}(\ell^{2})$
and $h^\epsilon:\mathbb{R}\rightarrow[0,1]$
such that $\mu_\tau^\epsilon(\cdot)= h^\epsilon(\tau)\mu^\epsilon(\cdot)$.
As before, by \eqref{-CC16}, we can verify
\begin{align}\label{iiii}
&\limsup_{k\rightarrow\infty}
\frac{1}{k+t}\int_{-k}^t
\mathbb{P}\bigg(\Big\{\omega\in\Omega:
u^\epsilon(t,\tau,x)\notin \mathcal{X}^\delta_N(t)
\Big\}\bigg) h^\epsilon(\tau)d\tau
\nonumber\\&\leq\sum_{l=1}^N
\frac{\delta}{2^{2l}C(t)}
\limsup_{k\rightarrow\infty}
\frac{1}{k+t}\int_{-k}^t
\mathbb{E}\big[\|\tilde{u}^\epsilon_{n^\delta_l(t)}
(t,\tau,x)
\|^2\big] h^\epsilon(\tau)d\tau
\nonumber\\&\ \ + \sum_{l=1}^N2^{2l}
\limsup_{k\rightarrow\infty}
\frac{1}{k+t}\int_{-k}^t
\mathbb{E}\big[\|\hat{u}^\epsilon_{n^\delta_l(t)}(t,\tau,x)\|^2\big]
h^\epsilon(\tau)d\tau
\leq\sum_{l=1}^N\frac{\delta}{2^{2l-1}}
\leq\frac{2\delta}{3}.
\end{align}
Then by Fatou's lemma, \eqref{uuua} and \eqref{iiii} we find
\begin{align*}
\eta_t^\epsilon(\mathcal{X}^\delta_N(t) )&
=\liminf_{k\rightarrow\infty}\int_{\ell^2} \frac{1}{k+t}\int_{-k}^t
\mathbb{P}\bigg(\Big\{\omega\in\Omega:
u^\epsilon(t,\tau,x)\in \mathcal{X}^\delta_N(t)
\Big\}\bigg) h^\epsilon(\tau)d\tau\mu^\epsilon(dx)
\nonumber\\&\geq1-\int_{\ell^2}
\limsup_{k\rightarrow\infty}\bigg[\frac{1}{k+t}
\int_{-k}^t
\mathbb{P}\bigg(\Big\{\omega\in\Omega:
u^\epsilon(t,\tau,x)\notin \mathcal{X}^\delta_N(t)
\Big\}\bigg) h^\epsilon(\tau)d\tau\bigg]
\mu^\epsilon(dx)
\geq\frac{\delta}{3}.
\end{align*}
This concludes  the proof of (i).

(ii) For every $\tau\in\mathbb{R}$, $t\geq\tau$
 $\delta>0$, $\epsilon_0\in\big[0,\frac{\sqrt{\lambda}}
{2\delta}\big]$ and compact set $\mathbf{K}$ of $\ell^2$, since we can prove that all uniform estimates of the solutions are uniform for $\epsilon\in\big[0,\frac{\sqrt{\lambda}}
{2\delta}\big]$, by a stoping time argument, we can prove
\begin{align}\label{yytr}
&\lim_{\epsilon\rightarrow\epsilon_0}
\sup_{u_0\in \mathbf{K}}
\mathbb{P}\bigg(\big\{\omega\in\Omega:
\|u^\epsilon(t,\tau,u_{0})-
u^{\epsilon_0}(t,\tau,u_{0})\|\geq \delta \big\}\bigg)=0.
\end{align}
The proof of \eqref{yytr} is quite similar to the autonomous case as in
\cite{Chenzhang1,Chenzhang2,LIding0,LIding1}, the details are omitted here.  Then by  (i), \eqref{yytr}, Theorem \ref{Main-results} and Prohorov's theorem, we complete the proof of (ii).
\end{proof}

\section{Acknowledgements}
Renhai Wang  was supported by China Postdoctoral Science Foundation under grant numbers
2020TQ0053 and 2020M680456.
Tom\'as Caraballo was supported by the Spanish Ministerio de Ciencia, Innovaci\'{o}n y Universidades project PGC2018-096540-B-I00,  and Junta de Andalucia (Spain) under projects US-1254251 and P18-FR-4509. Nguyen Huy Tuan was  supported by Van Lang University.

%\section{Reference}

\end{document}